\documentclass{amsart}
\usepackage{chngcntr}
\usepackage{apptools}
\usepackage{color}
\usepackage{amsthm,amssymb,verbatim}
\usepackage{graphicx}
\usepackage{enumerate}
\usepackage{chngcntr}
\usepackage{apptools}
\usepackage{color}
\usepackage{amsthm,amssymb,verbatim}
\usepackage{graphicx}
\usepackage{enumerate}
\usepackage{hyperref}
\usepackage[alphabetic, msc-links, backrefs]{amsrefs}

\newcommand{\Ff}{\mathcal F}

 \newcommand{\Cc}{\mathcal C}
 
 \newcommand{\Ee}{\mathcal{E}}

  \newcommand{\Ss}{\mathbf{S}}
 
 \newcommand{\RR}{\mathbf{R}}  
 \newcommand{\BB}{\mathbf{B}}  

    \newcommand{\dist}{\operatorname{dist}}
 
 \newcommand{\area}{\operatorname{area}}
 \newcommand{\eps}{\epsilon}
 \newcommand{\Tan}{\operatorname{Tan}}

\newcommand{\ee}{\mathbf e}

\newcommand{\Hh}{\mathcal{H}}
\usepackage{amsthm}

\input epsf
\def\begfig {
\begin{figure}
\small }
\def\endfig {
\normalsize
\end{figure}
}

    \newtheorem{theorem}    {Theorem}   
    \newtheorem{lemma}      [theorem]       {Lemma}
    \newtheorem{corollary}  [theorem]     {Corollary}
    \newtheorem{proposition}       [theorem]       {Proposition}
    
    \newtheorem{claim}{Claim}
    \newtheorem*{claim*}{Claim}
    \newtheorem*{theorem*}{Theorem}
    \theoremstyle{definition}
    \newtheorem{definition}  [theorem] {Definition}
     
    \theoremstyle{definition}
    \newtheorem{remark}   [theorem]       {Remark}
    \newtheorem*{remark*} {Remark}

\title[Boundary Singularities]{Boundary Singularities in Mean Curvature Flow and Total Curvature of Minimal Surface Boundaries}

\author{Brian White}

\date{June 12, 2021.  Revised August 10, 2022}

\begin{document}

\begin{abstract}
For hypersurfaces moving by standard mean curvature flow with fixed boundary, we show that if a tangent flow at a boundary singularity is given by a smoothly embedded shrinker, then the shrinker must be non-orientable.  We also show that
there is an initially smooth surface in $\RR^3$ that develops a boundary singularity for which the shrinker is smoothly embedded (and therefore non-orientable).  Indeed, we show that there is a nonempty open set of such initial surfaces.

Let $\kappa$ be the largest number with the following property: if $M$ is a minimal surface in $\RR^3$ bounded by a smooth simple
closed curve of total curvature $< \kappa$, then $M$ is a disk.  Examples show that $\kappa<4\pi$.  In this paper, we use mean curvature flow to show that $\kappa \ge 3\pi$.  We get a slightly larger lower bound for orientable surfaces.
\end{abstract}

\maketitle

\tableofcontents

\section{Introduction}

Suppose $M\subset \RR^{n+1}$ is a compact, smoothly embedded $n$-manifold with boundary $\Gamma$.
According to~\cite{white-mcf-boundary}, there is a standard Brakke flow
\[
   t\in [0,\infty) \mapsto M(t)
\]
with (fixed) boundary $\Gamma$ and with initial surface $M(0)=M$.  
Furthermore, if $\Gamma$ lies on the boundary of the convex hull of $M$,
 then
the flow is regular at the boundary for all times.

(See~\cite{white-mcf-boundary}*{5.1 and~13.2} for the definitions of Brakke flow with boundary and standard Brakke flow with boundary.  Briefly, standard Brakke flows are those that are unit-regular -- which prevents certain gratuitous vanishing -- and that take their boundaries as mod $2$ chains.  In particular, triple junction singularities do not occur in standard Brakke flow.)

In this paper, we explore boundary regularity of such a flow $M(\cdot)$ without assuming that $\Gamma$ 
lies on the boundary of the convex hull of $M$.

In particular, we show that if a shrinker corresponding to a boundary singularity is smooth and embedded, then
it must be non-orientable.
We also show that smooth, non-orientable shrinkers do arise as boundary singularities of certain smooth initial  surfaces in $\RR^3$.  Indeed, we show  that there is a nonempty open set of such initial surfaces.
See Theorem~\ref{unavoidable-theorem}.

We also apply mean curvature flow to questions in minimal surface theory:
\begin{enumerate}
\item What is the largest number $\kappa$ such that if $M$ is a smooth minimal surface in $\RR^3$ bounded by a smooth simple closed curve of total curvature $< \kappa$, then $M$ must be a disk?
\item What is the largest number $\kappa'$ such that if $M$ is an orientable smooth minimal surface in $\RR^3$ bounded by a smooth simple closed curve of total curvature $< \kappa'$, then $M$ must be a disk?
\end{enumerate}

Examples show that $\kappa< 4\pi$.  Here, we show (Theorem~\ref{3-pi-theorem}) that $\kappa>3\pi$.

Examples of Almgren-Thurston show that $\kappa'\le 4\pi$.  It is conjectured that $\kappa'=4\pi$.
We show (Theorem~\ref{sharper}) that $\kappa' > (2\pi)^{3/2}e^{-1/2} \sim 3\pi (1.014)$.  

For the existence results in this paper, we work with standard Brakke flows of two-dimensional surfaces with
entropy less than two.  Such flows are rather well-behaved: they are smooth at almost all times, and the shrinkers corresponding to tangent flows are smoothly embedded and have multiplicity one.  See Theorem~\ref{long-theorem}.

\section{A Bernstein Theorem for Orientable Boundary Shrinkers}

\begin{theorem}\label{bernstein-theorem}
Let $M\subset \RR^{n+1}$ be a smoothly embedded, 
oriented shrinker bounded by an $(n-1)$-dimensional
linear subspace $L$.  Then $M$ is a half-plane.
\end{theorem}

\begin{proof}
We can assume that $L=\{x: x_1=x_2=0\}$.  Consider the $1$-form
\[
    d\theta = \frac{x_1\,dx_2 - x_2\,dx_1}{x_1^2 +x_2^2}.
\]
on $\RR^{n+1}\setminus L$.  

Let $C$ be an oriented closed curve in $M\setminus L$.  
The winding number of $C$ about $L$ is equal to the intersection number of $C$ and $M$.
Since we can move $C$ slightly in the direction of the unit normal to $M$ to get a curve $C'$ disjoint from $M$, we see that the winding number of $C$ about $L$ is equal to $0$.  Thus
\[
   \int_C d\theta = 0.
\]
Since this holds for any closed curve in $M\setminus L$, we see that $d\theta$ is exact on $M$.
Thus there is a single-valued function
\[
  \theta: M\setminus L \to \RR
\]
such that for $x\in M\setminus L$, 
\begin{align*}
 x_1 &= (x_1^2+x_2^2)^{1/2} \cos\theta(x), \\
 x_2 &= (x_1^2+x_2^2)^{1/2} \sin\theta(x).
\end{align*}
Note that $\theta(\cdot)$ extends continuously to $L$.

Let $M'$ be a connected component of $M$.

\begin{claim}\label{maximum-claim}
 $\theta|M'$ attains a maximum.
\end{claim}

\begin{proof}[Proof of claim~\ref{maximum-claim}]
Let $g$ be the shrinker metric on $\RR^{n+1}$.
Choose an $R$ large enough that the $g$-mean curvature vector of $\partial \BB(0,R)$ 
points outward.  (That is, choose $R>\sqrt{2n}$.)

Let $\alpha$ be the maximum of $\theta$ on $\{x\in M': |x|\le 3R\}$.  By rotating, we can assume that $\alpha=0$.

Suppose the claim is not true.  Then 
\[
    \Sigma:=\{x\in M': 0<\theta(x)< \pi\}
\]
 is a nonempty hypersurface in $\{x_2>0\}$ whose boundary lies in the hyperplane $\{x_2=0\}$.

Consider $n$-dimensional surfaces $S$ in $\{x: x_2\ge 0\}$ such that
\begin{enumerate}
\item $\partial S =  \partial \BB(0,2R) \cap \{x_2=0\}$.
\item $S$ is rotationally invariant about the $x_2$-axis.  (That is, if $\rho\in O(n+1)$ and $\rho(\ee_2)=\ee_2$, then
   $\rho(S)=S$.)
\item The region $U$ bounded by $S$ and $\BB(0,2R)\cap\{x_2=0\}$ contains $\BB(0,R)\cap\{x_2\ge 0\}$ and is 
   disjoint from $\Sigma$.
\end{enumerate}
Choose a surface $S$ that minimizes $g$-area with among all such surfaces.  Then $S$ is smooth, $g$-minimal,
and $g$-stable.  

(The restriction that $S$ be rotationally invariant ensures that $S$ is smooth; otherwise, $S$ might have an $n-7$-dimensional
singular set.   If the smoothness is not clear, note that the rotational symmetry implies that the quotient set
\[
  \{ ( |(x_1,0, x_3, \dots, x_{n+1})|, x_2) : x\in S \}
\]
is a curve; that curve must be a geodesic with respect to a certain Riemannian metric on $(0,\infty)\times\RR$.)

According to~\cite{brendle}*{Proposition~5}, $S$ is flat with respect to the Euclidean metric.
(One lets $k\to\infty$ in the statement of that proposition.)
But that is impossible since $S\setminus \partial S\subset \{x_2>0\}$ and since $\partial S$ is an $(n-1)$-sphere in $\{x_2=0\}$.
This completes the proof of the claim.
\end{proof}

By Claim~\ref{maximum-claim} and the strong maximum principle, $\theta$ is constant on $M'$.  Thus $M'$ is a half-plane with boundary $L$.
Consequently, $M$ is a union of such half-planes.  Since $M$ is embedded, it is a single half-plane.
\end{proof}

\begin{remark*} 
The proof of Theorem~\ref{bernstein-theorem} was inspired by the proof of Theorem 11.1 in the Hardt-Simon boundary regularity paper~\cite{hardt-simon-boundary}.
\end{remark*}

\section{A Bernstein Theorem for Orientable Boundary Shrinkers with Singularities}

In this section, we extend Theorem~\ref{bernstein-theorem} to possibly singular shrinkers.  
The proof is not longer, but it does use more machinery.
This section is
not used in the rest of the paper.

\newcommand{\sing}{\operatorname{sing}}
\newcommand{\reg}{\operatorname{reg}}
\begin{theorem}\label{bernstein2}
Let $M\subset \RR^{n+1}$ be a
shrinker bounded by an $(n-1)$-dimensional
linear subspace $L$.   Suppose that 
\begin{equation}\label{nullset}
 \Hh^{n-1}(\sing M) = 0
\end{equation}
and that the regular part 
\[
\reg M:=M\setminus \sing M
\]
 is orientable. Then $M$ is a half-plane.
\end{theorem}

(By definition, $\reg M$ is a smooth, properly embedded manifold-with-boundary
 in $\RR^{n+1}\setminus (\sing M)$, the boundary being $L\setminus \sing M$.)

\begin{proof}
As in the proof of Theorem~\ref{bernstein-theorem}, we assume that $L$ is the subspace 
  $\{x_1=x_2=0\}$.  As in that proof, there is a smooth function
\[
  \theta: \reg M \to \RR
\]
such that
\begin{align*}
 x_1 &= (x_1^2+x_2^2)^{1/2} \cos\theta(x), \\
 x_2 &= (x_1^2+x_2^2)^{1/2} \sin\theta(x).
\end{align*}

\begin{claim}\label{connected-claim}
Let $\mathcal{H}$ be the halfspace $ax_1+bx_2>0$, where $a$ and $b$ are not both zero.
Then $\reg(M)\cap \mathcal{H}$ is connected.
\end{claim}

This is stated for shrinkers without boundary in~\cite{choi-h-h-white}*{Corollary~3.12}.
  But the same proof gives Claim~\ref{connected-claim}.

Let $M'$ be a connected component of $\reg(M)$.  Thus $\theta(M')$ is an interval.

First, we claim that $I$ has length at most $\pi$.  For if not, by rotating we can assume that 
\[
  \inf_{M'} \theta< 0
\]
and that
\[
 \sup_{M'}\theta> \pi.
\]
But then $\{x\in M': \theta(x)\in (-\pi,0)\}$ and $\{x\in M': \theta(x) \in (\pi,2\pi)\}$ are nonempty and lie
in different components of $M'\cap \{x_2<0\}$, contradicting Claim~\ref{connected-claim}.

Next, we claim that $M'$ is a half-plane.  For suppose not.  Then $\theta|M'$ does not attain a maximum
or a minimum by the strong maximum principle.  Thus the interval $I=\theta(M')$ is an open interval
of length at most $\pi$.  By rotating, we can assume that
\[
  \theta(M') = (-\alpha,\alpha)
\]
for some $\alpha$ with $0<\alpha\le \pi/2$.  Now rotate $M'$ by $\pi$ about $L$ to get $M''$. Then
$M^*:=\overline{M'\cup M''}$ is an embedded shrinker without boundary.  Furthermore, $\reg(M^*)\cap \{x_2>0\}$ is 
not connected, which is impossible by Claim~\ref{connected-claim} (applied to~$M^*$).

We have shown that each component of $\reg(M)$ is a half-plane bounded by $L$.  
By~\eqref{nullset}, there can be at most one such half-plane.
\end{proof}

\section{Basic Properties of Entropy and Total Curvature}

Suppose that $\Gamma$ is an $(m-1)$-dimensional submanifold of $\RR^n$ and that $v\in \RR^n$.  We define the cone $C_{\Gamma,v}$ and the exterior cone $E_{\Gamma,v}$ over $\Gamma$ with vertex $v$ by
\begin{align*}
C_{\Gamma,v} &= \{ v + s(x-v): x\in \Gamma, \, s\in [0,\infty)\}, \\
E_{\Gamma,v} &= \{ v + s(x-v): x\in \Gamma, \, s\in [1,\infty).
\end{align*}

We will also use $C_{\Gamma,v}$ and $E_{\Gamma,v}$ to denote the corresponding Radon measures on $\RR^n$
(counting multiplicity).  Thus
\begin{align*}
C_{\Gamma,v} f  &=  \int_{p\in \RR^n} f(p)\, \Hh^0\{ (x,s)\in \Gamma\times [0,\infty): v+s(x-v)=p\} \,d\Hh^m p, \\
E_{\Gamma,v} f  &=  \int_{p\in \RR^n} f(p)\, \Hh^0\{ (x,s)\in \Gamma\times [1,\infty): v+s(x-v)=p\} \,d\Hh^m p.
\end{align*}
Now $C_{\Gamma,v}$ and $E_{\Gamma,v}$ depend continuously on $v$ for $v\in \RR^n\setminus \Gamma$.
However, the dependence is not continuous at $v\in \Gamma$.  For suppose $v_i\notin \Gamma$ converges
to $v\in \Gamma$.
Then, after passing to a subsquence,  
  $C_{\Gamma,v_i}$ converges to the union of $C_{\Gamma,v}$ and  a half-plane
bounded by $\Tan(\Gamma,v)$, and similarly for $E_{\Gamma,v_i}$.

Consequently,
\begin{equation}\label{cone-quantity}
   v\in \RR^n \mapsto  \frac{C_{\Gamma,v}\BB(v,r)}{\omega_m r^m} + \frac12 1_{\Gamma}(v)
\end{equation}
is continuous in $v$.  (Since $C_{\Gamma,v}$ is conical, the quantity~\eqref{cone-quantity} is independent of $r$.)
Here $1_\Gamma(\cdot)$ is the characteristic function (or indicator function) of $\Gamma$.
\newcommand{\vis}{\operatorname{vis}}

We define the {\bf vision number} $\vis(\Gamma)$ of $\Gamma$ to be the supremum of the
 quantity~\eqref{cone-quantity} 
 over $v\in \RR^n$.  If $\Gamma$ is compact, then the supremum is attained by a $v$ in the convex hull of $\Gamma$.

\begin{remark}
Because the quantity~\eqref{cone-quantity} depends continuously on $v$, 
\begin{align*}
\sup_v \left(\frac{C_{\Gamma,v}\BB(v,r))}{\omega_m r^m} + \frac12 1_{\Gamma}(v) \right)
&=
\sup_{v\notin \Gamma} \left( \frac{C_{\Gamma,v}\BB(v,r))}{\omega_m r^m} + \frac12 1_{\Gamma}(v) \right) \\
&=
\sup_{v\notin\Gamma}\frac{C_{\Gamma,v}\BB(v,r))}{\omega_m r^m} \\
&\le
\sup_v \frac{C_{\Gamma,v}\BB(v,r))}{\omega_m r^m}  \\
&\le
\sup_v \left(\frac{C_{\Gamma,v}\BB(v,r))}{\omega_m r^m} + \frac12 1_{\Gamma}(v) \right).
\end{align*}
Since the first and last expressions are the same, in fact equality holds.
Thus in the definition of vision number, it does not matter whether or not we include the term $\frac121_\Gamma(p)$.
\end{remark}

\begin{definition}\label{entropy-def}
If $M$ is a Radon measure on $\RR^n$ and if $\Gamma$ is a smooth $(m-1)$-dimensional manifold in $\RR^{n}$,  then
the {\bf entropy} of the pair $(M,\Gamma)$ is
\[
e(M;\Gamma)
=
\sup_{v\in \RR^n, \lambda>0} (M + E_{v,\Gamma}) \psi_{v,\lambda},
\]
where
\[
 \psi_{v,\lambda}(x) = \frac1{(4\pi \lambda)^{m/2}} \exp\left( - \frac{|x-v|^2}{4\lambda} \right).
\]
\end{definition}

\begin{theorem}[\cite{white-mcf-boundary}, Theorem~8.1]
Suppose that
\[
   t\in I \mapsto M(t) 
\]
is a Brakke flow with boundary $\Gamma$ in $\RR^n$.  Then $e(M(t);\Gamma)$ is a decreasing function of $t$.
\end{theorem}

\begin{theorem}\label{entropy-vision-theorem}
Suppose that $\Gamma$ is a smooth $(m-1)$-dimensional manifold in $\RR^n$ and that
$t\in [T_0,T] \mapsto M(t)$ is an $m$-dimensional Brakke flow with boundary $\Gamma$.
Then
\[
   e(M(t);\Gamma) \le \frac{\area(M(T_0))}{(4\pi (t-T_0))^{m/2}} + \vis(\Gamma).
\]
\end{theorem}

\begin{proof}
This is an immediate consequence of~\cite{white-mcf-boundary}*{Corollary~7.3}.
\end{proof}

\begin{corollary}\label{entropy-vision-corollary}
In Theorem~\ref{entropy-vision-theorem}, if $M(\cdot)$ is an ancient flow and if the area of $M(t)$ is bounded above as $t\to -\infty$,
then
\[
  e(M(t);\Gamma) \le \vis(\Gamma).
\]
\end{corollary}

\newcommand{\tc}{\operatorname{tc}}

\begin{theorem}[\cite{EWW}*{Theorem~1.1}]\label{vision-tc-theorem}
Suppose that $\Gamma$ is a simple closed curve in $\RR^n$.  Then
\[
   \vis(\Gamma) \le \frac1{2\pi}\tc(\Gamma),
\]
where $\tc(\Gamma)$ is the total curvature of $\Gamma$.  Equality holds if and only if $\Gamma$ is a convex planar curve.
\end{theorem}

\begin{theorem}\label{long-theorem}
Suppose that $\Gamma$ is a smooth, simple closed curve in $\RR^n$. 
Suppose that 
\[
   (a,\infty)\in \RR \mapsto M(t)
\]
is a standard Brakke flow with boundary $\Gamma$ such that 
\[
  A:=\sup_{t\in (a,\infty)} \area(M(t))=\lim_{t\to a}\area(M(t)) < \infty,
\]
and that
\[
  \sup_t e(M(t);\Gamma) < 2.
\]
Then
\begin{enumerate}[\upshape (1)]
\item\label{shrinker-smooth}
At each spacetime point, each tangent flow is given by a smooth, multiplicity $1$ shrinker.
\item\label{most-times} The flow $M(\cdot)$ is regular at almost all times.  
\item\label{shrinker-non-orientable}
If $n=3$, then for each tangent flow
at a boundary singularity, 
the corresponding shrinker is non-orientable.  
\item\label{future} Every sequence of times tending to $\infty$ has a subsequence $t(i)$ for which $M(t(i))$ converges smoothly
to an embedded minimal surface.
\item\label{unique-future} As $t\to\infty$, $M(t)$ converges smoothly to an embedded minimal surface $M(\infty)$.
\item\label{unique-past} If $a=-\infty$, then as $t\to -\infty$, $M(t)$ converges smoothly to an embedded 
 minimal surface $M(-\infty)$.
\end{enumerate}
\end{theorem}

\begin{proof}
See Appendix~\ref{regularity-appendix} for the proofs of Assertions~\ref{shrinker-smooth} and~\ref{most-times}.
In particular, Lemmas~\ref{plane-lemma} and~\ref{half-plane-lemma} give Assertion~\ref{shrinker-smooth}, and
Proposition~\ref{most-times-proposition} gives Assertion~\ref{most-times}.

In Assertion~\ref{shrinker-non-orientable}, the non-orientability of $\Sigma$ follows immediately
 from Assertion~\ref{shrinker-smooth} and Theorem~\ref{bernstein-theorem}.

To prove Assertion~\ref{future}, suppose $t(i)\to \infty$. By passing to a subsequence, we may assume that the flows
\[
   t\mapsto M(t(i)+t)
\]
converge to a standard limit Brakke flow
\[
  t\in \RR \mapsto M'(t)
\]
with boundary $\Gamma$. 
(This is by the compactness theory for standard Brakke flows with boundary: see Theorems~10.2, 10.2, and~13.1 and Definition~13.2 in 
\cite{white-mcf-boundary}.)  
 Since the area of $M'(t)$ is independent of $t$ (it is equal to $\lim_{t\to\infty}\area(M(t))$),
it follows that $M'(t)$ is a stationary integral varifold $V$ independent of $t$.  
By Assertion~\ref{most-times} applied to the flow~$M'(\cdot)$, the surface $M'(t)$ is smoothly embedded for almost all $t$.  Thus $V$ is smoothly embedded.

The uniqueness in Assertion~\ref{unique-future} follows from Assertion~\ref{future} 
and from the Lojasiewicz-Simon inequality~\cite{simon-asymptotics}*{Theorem~3}.
See Theorem~\ref{unique} below.

The proof of Assertion~\ref{unique-past} is the same as the proof of Assertion~\ref{unique-future}.
\end{proof}

\section{A general existence theorem for eternal flows}

In the following two theorems, 
$\Ff$ is a set of $C^1$ compact manifolds-with-boundary in $\RR^3$ with the following property:
if $M\in \Ff$ and if $M'$ is isotopic to $M$, then $M'\in \Ff$.
 For example, $\Ff$ might be the set
of non-orientable surfaces, or the set of genus-one orientable surfaces, etc.

\begin{theorem}\label{general-existence-theorem}
Suppose that $\Gamma$ is a smooth, simple closed curve in $\RR^3$ of total curvature at most $4\pi$.
Suppose that $\Gamma$ bounds a minimal surface in $\Ff$ and that $\Gamma$ is a smooth limit of curves
$\Gamma_i$ such that $\Gamma_i$ does not bound a minimal surface in $\Ff$.
 
Then there is an eternal standard Brakke flow 
\[
  t\in \RR \mapsto M'(t)
\]
with boundary $\Gamma$ such that 
\begin{enumerate}[\upshape (i)]
\item\label{entropy-item} 
\[
   \sup_te(M'(t);\Gamma)< \frac1{2\pi} \tc(\Gamma) \le 2.
\]
\item $M'(t)$ converges smoothly as $t\to -\infty$ to a minimal surface $M'(-\infty)$ in $\Ff$.
\item $M'(-\infty)$ has the least area among all minimal surfaces of type $\Ff$ bounded by $\Gamma$.
\item $M'(t)$ converges smoothly as $t\to \infty$ to a minimal surface $M'(\infty)$ that is not in $\Ff$.
\end{enumerate}
\end{theorem}

\begin{proof} 
The set of minimal surfaces in $\Ff$ bounded by $\Gamma$ is compact
(by Theorem~\ref{minimal-compact}), so there
is a surface $M$ that attains the least area $A$ among all such surfaces.

Note that $t\in\RR\mapsto M$ is a standard Brakke flow, so
\[
  e(M;\Gamma) \le \vis(\Gamma) < \frac1{2\pi}\tc(\Gamma) \le 2
\]
by Corollary~\ref{entropy-vision-corollary} and Theorem~\ref{vision-tc-theorem}.

For each $i$, let $M_i$ be a surface diffeomorphic to $M$ and bounded by $\Gamma_i$ such that $M_i$ converges
smoothly to $M$ as $i\to\infty$. Then $e(M_i;\Gamma_i)\to e(M;\Gamma)$, so by passing to a subsequence, we can assume that
\[
   e(M_i;\Gamma_i) < 2.
\]
for all $i$.

Let 
\[
  t\in [0,\infty) \mapsto M_i(t)
\]
be a standard Brakke flow with boundary $\Gamma_i$ and with initial surface $M_i(0)=M_i$.

Of course,
\begin{equation}\label{entropy-sup}
\sup_t e(M_i(t);\Gamma_i) = e(M_i;\Gamma_i)  < 2,
\end{equation}
and
\begin{equation}\label{area-sup}
   \sup_t \area(M_i(t)) = \area(M_i(0))  = \area(M_i) \to \area M = A.
\end{equation}

As $t\to\infty$, we can assume, by passing to a subsequence, that the flow $M_i(\cdot)$
converges to a standard Brakke flow $M(\cdot)$ with $M(0)=M$. Since $M$ is smooth and minimal, the only such flow is the constant flow $M(t)\equiv M$.

As $t\to \infty$, $M_i(t)$ converges smoothly to an embedded minimal surface $M_i(\infty)$ (by~\eqref{entropy-sup} and
Theorem~\ref{long-theorem}).    By hypothesis on $\Gamma_i$, the surface $M_i(\infty)$ is not in the collection $\Ff$.
Thus for all sufficiently large $t$, $M_i(t)$ is not in $\Ff$.

Since $M_i(0)$ is in $\Ff$ and since $M_i(t)$ is not in $\Ff$ for large $t$, we see that there must
be singularities in the flow.  Let $T_i$ be the first singular time.
Since the flow $t\mapsto M_i(t)$ converges to the constant flow $t\mapsto M$, we see that
\[
  T_i \to \infty.
\]

By passing to a subsequence, we can assume that the time-shifted flows
\[
   t\in [-T_i,\infty) \mapsto M_i'(t):= M_i(T_i + t)
\]
converge to an eternal standard Brakke flow
\[
  t\in \RR \mapsto M'(t)
\]
with boundary $\Gamma$.   
 By Corollary~\ref{entropy-vision-corollary} and Theorem~\ref{vision-tc-theorem}, the flow satisfies the entropy bound~\eqref{entropy-item}.

By Theorem~\ref{long-theorem}, there is a $c\in (0,\infty)$ such that
\begin{align}
&\text{$M'(t)$ is smooth for $|t|>c$}, \\
&\text{$M'(t)$ converges smoothly to a minimal surface $M'(-\infty)$ as $t\to - \infty$, and} \\
&\text{$M'(t)$ converges smoothly to a minimal surface $M'(\infty)$ as $t\to  \infty$}.
\end{align}

By choice of $T_i$, $M_i'(t)$ in $\Ff$ for $t<0$.  
By  local regularity~\cite{white-local}, $M_i'(t)$ converges smoothly
to $M'(t)$ for $t< -c$.  Thus $M'(t)$ is in $\Ff$ for $t<-c$, and hence $M'(-\infty)$ is in $\Ff$.

Since $0$ is a singular time of the flow $M_i'(\cdot)$,
 it follows (again by local regularity~\cite{white-local}) that $0$ is a singular
time of the flow $M'(\cdot)$.  Hence the flow $M'(\cdot)$ is not constant, so 
\begin{align}
\area(M'(\infty)) 
&<
\area(M'(-\infty)) \le A.
\end{align}

Since $A$ is the least area of any minimal $\Ff$-type surface bounded by $\Gamma$, 
it follows that $\area M'(-\infty)=A$ and that $M'(\infty)$ is not of type $\Ff$.
\end{proof}

\begin{theorem}\label{general-redux}
Suppose that the family $\Ff$ does not include disk-type surfaces.
Suppose there is a smooth simple closed curve $\Gamma_0$ of total curvature $< 4\pi$ that bounds a minimal
surface of type $\Ff$.  
Then there is a curve $\Gamma$ such that $\tc(\Gamma)\le \tc(\Gamma_0)$ and such that $\Gamma$ satisfies the hypotheses of Theorem~\ref{general-existence-theorem}.
\end{theorem}

\begin{proof}
By Theorem~\ref{deformation-theorem}, there is a smooth one-parameter family
\[
    s\in [0,1]\mapsto \Gamma_s
\]
of simple closed curves (starting from the given curve $\Gamma_0$) for which $\Gamma_1$ is a round circle
and for which each curve has total curvature $\le\tc(\Gamma_0)$.

Let $S$ be the set of $s\in [0,1]$ such that $\Gamma_s$ bounds a minimal surface of type $\Ff$.

By Theorem~\ref{minimal-compact}, $S$ is closed. And $S$ is nonempty since $0\in S$.  Thus $S$ has a maximum $\hat{s}$.  Note that $\hat{s}<1$.
Hence $\Gamma_{\hat{s}}$ bounds a minimal surface of type $\Ff$, and $\Gamma_{\hat{s}}$ is a smooth 
limit of the curves $\Gamma_s$, $s>\hat{s}$, that do not bound minimal surfaces of type $\Ff$.
\end{proof}

\section{Generalized M\"obius Strips}

Suppose that $M\subset \RR^3$ is a smoothly embedded, compact surface with exactly
one boundary component.  More generally, we can allow $M$ to have self-intersections
and singularities, provided they occur away from the boundary.  
That is, we only require that $\partial M$ is a smooth, simple closed curve and that $M$
is a smoothly embedded manifold-with-boundary (the boundary being $\partial M$) near $\partial M$.

Although $M$ may not be orientable, we can choose an orientation for $\partial M$.
Now push $\partial M$ slightly into $M$ to get another smooth embedded curve $C$.
  For example, if $\eps>0$ is sufficiently small, we can let
\[
  C = \{p\in M: \dist(p,\partial M)=\eps\}.
\]

We let $\lambda(M)$ be the linking number of $\partial M$ and $C$.  
This can be defined in various (equivalent) ways.  For example, let $\Sigma$ be 
a compact oriented surface (not necessarily embedded) with boundary $\partial M$.
By perturbing $\Sigma$ slightly, we can assume that $C$ intersects $\Sigma$ transversely.  Then $\lambda(M)$ is the intersection number 
(in $\RR^3\setminus \partial M$) of 
$C$ and $\Sigma$ 

Alternatively, we can let $\Sigma$ be a compact oriented surface with boundary $C$.
Then $\lambda(M)$ is the intersection number of $\partial M$ and $\Sigma$ in 
  $\RR^3\setminus C$.

(We began by choosing an orientation for $\partial M$, but the 
resulting value of $\lambda(M)$
 does not depend on that choice, since reversing the orientation of $\partial M$ also 
 reverses the orientations of $C$ and of $\Sigma$, thus leaving the value of 
$\lambda(M)$ the same.)

Of course if $M$ is smoothly embedded and orientable, then $\lambda(M)=0$, since we can let $\Sigma$
be the portion of $M$ bounded by $C$.  (Note that $\Sigma$ is an oriented surface
with boundary $C$ and that $\Sigma$ is disjoint from $\partial M$, so the linking number of $C$ and $\partial M$ is $0$.)

\begin{proposition}
If $M$ is a smoothly embedded M\"obius strip, then $\lambda(M)\ne 0$.
Indeed, $\lambda(M)/2$ is an odd integer.
\end{proposition}

\begin{proof}
Let $S$ be an embedded, orientation-reversing path in the interior of $M$. 
Note that we can perturb $S$ slightly to get a curve $S'$ that intersects $M$ transversely and in exactly one point.  Thus the mod $2$ linking number of $\partial M$ and $S$ is $1$, and therefore the integer linking number
of $\partial M$ and $S$ (whichever way we orient those curves) is an odd integer.

Now push $\partial M$ into $M$ to get an embedded curve $C$ as in the definition 
of $\lambda(M)$.  Then $C$ is homologous in $M\setminus \partial M$ (and therefore 
in $\RR^3\setminus \partial M$) to $S$ traversed twice, so $\lambda(M)$, the linking number of $\partial M$ and $C$, is equal to the twice the linking number of 
 $\partial M$ and $S$.
 \end{proof}

\begin{definition}
A {\bf generalized M\"obius strip} in $\RR^3$ is a smoothly embedded, compact
(not necessarily connected) surface $M$  in $\RR^3$ such that
$\partial M$ has exactly one component, and such that $\lambda(M)$ is nonzero.
\end{definition}

Every generalized M\"obius strip is non-orientable, but not every non-orientable surface
in $\RR^3$ is a generalized M\"obius strip.   For example, if we attach a handle to a flat disk to make a non-orientable surface $M$ (topologically a Klien bottle with a disk removed), then $\lambda(M)=0$.

Note that if a generalized M\"obius strip in Euclidean space is a minimal surface, then it has no components without boundary and thus must be connected.

\begin{lemma}\label{bdry-forcing-lemma}
Suppose $\Gamma$ is a smooth, simple closed curve in $\RR^3$,
suppose $t\in [a,b]\mapsto M(t)$ is a Brakke flow with boundary $\Gamma$ such that $a$ and $b$ are regular times,
and suppose the flow has  no boundary singularities.  
Then $t\in [a,b]\mapsto \lambda(M(t))$ is constant.
In particular, $M(a)$ is a generalized Mobius strip if and only if $M(b)$ is a generalized Mobius strip.
\end{lemma}

\begin{proof} 
The lemma is true because $\lambda(M(t))$ only depends on the behavior of $M(t)$ in an arbitrarily small neighborhood of the boundary.
\end{proof}

\begin{theorem}\label{mobius-theorem}
Let $\kappa_{gm}$ be the infimum of the total curvature $tc(\Gamma)$ among smooth, simple
closed curves $\Gamma$ in $\RR^3$ that bound generalized minimal M\"obius strips.
Then
\begin{equation}
\kappa_{gm} < 4\pi.
\end{equation}
If $\kappa_{gm} < \kappa < 4\pi$, then there exists a smooth, simple closed curve $\Gamma$ in $\RR^3$ of total
curvature $<\kappa$ and an eternal standard Brakke flow $t\in \RR\mapsto M(t)$ with boundary $\Gamma$ such that
\begin{enumerate}[\upshape (1)]
\item\label{first-property} $\sup_te(M(t); \Gamma) < \dfrac{\tc(\Gamma)}{2\pi}$.
\item $M(t)$ converges smoothly as $t\to -\infty$ to a generalized M\"obius strip $M(-\infty)$.
\item The surface $M(-\infty)$ has the least area  of any minimal, generalized
    M\"obius strip bounded by $\Gamma$.
\item\label{last-property} $M(t)$ converges smoothly as $t\to\infty$ to a minimal surface $M(\infty)$ that is not a generalized M\"obius strip.
\end{enumerate}
Furthermore, such a flow must have a boundary singularity, and the shrinker $\Sigma$ corresponding to a tangent
flow at any such boundary singularity is a smoothly embedded, nonorientable surface with straight line boundary.
\end{theorem}

\begin{proof}
By~\cite{EWW}*{\S5}, there is a smooth, simple closed curve of total curvature $<4\pi$ such that the curve bounds a minimal 
Mobius strip.  Thus $\kappa_{gm}<4\pi$.

By Theorem~\ref{general-redux}, there exists a curve $\Gamma$ and an eternal standard Brakke flow $M(\cdot)$ with boundary $\Gamma$
and having properties~\eqref{first-property}--\eqref{last-property}.  (One lets $\Ff$ be the family of all
generalized Mobius strips.)

By Lemma~\ref{bdry-forcing-lemma}, there must be a boundary singularity.   
By Theorem~\ref{long-theorem}, the tangent flow must have the indicated properties.
\end{proof}

\section{Boundary singularities are unavoidable}

\begin{theorem}\label{unavoidable-theorem}
Let $\Cc$ be the set of smoothly embedded surfaces $M$ in $\RR^3$ such that:
\begin{enumerate}[\upshape (i)]
\item the boundary curve $\Gamma$ has total curvature less than $4\pi$,
\item the entropy $e(M;\Gamma)$ is less than $2$,
and
\item any standard Brakke flow with boundary $\Gamma$ and with initial surface $M$ must develop a boundary singularity for which the corresponding shrinker is a smoothly embedded, multiplicity-one, nonorientable surface with straight line boundary.
\end{enumerate}
Then $\Cc$ has nonempty interior.
\end{theorem}

\begin{proof}
Let $\Gamma$ and 
\[
  t\in \RR \mapsto M(t)
\]
be as in Theorem~\ref{mobius-theorem}.   Let $T$ be a regular time at which $M(T)$ is a smoothly embedded generalized M\"obius strip.
(For example, $T$ could be any time before the first singularity.)
    We will show that $M(T)$ lies in the interior of $\Cc$.

By time translation, we can assume that $T=0$.
Let $M_i$ be a sequence of surfaces that converge smoothly to $M(0)$.  Let $\Gamma_i=\partial M_i$.
Let $t\in [0,\infty)\mapsto M_i(t)$ be a standard Brakke flow with boundary $\Gamma_i$ and with initial 
surface $M_i(0)=M_i$.

Note that
\begin{equation}\label{entropy<2}
   \sup_t e(M_i(t); \Gamma_i) = (e(M_i(0); \Gamma_i) \to e(M(0);\Gamma)< 2.
\end{equation}

Consequently, $M_i(t)$ converges smoothly as $t\to\infty$ to an embedded minimal surface $M_i(\infty)$.
By passing to a subsequence, we can assume (see Theorem~\ref{minimal-compact}) that $M_i(\infty)$ converges smoothly to a minimal surface $M'$ bounded by $\Gamma$.    Now
\[
  \area(M') = \lim_i \area(M_i(\infty)) \le \lim_i \area M_i(0) = \area M(0) < \area(M(-\infty)).
\]
Since $M(-\infty)$ achieves the least area of minimal generalized M\"obius strips bounded by $\Gamma$, 
it follows that $M'$ is not a generalized M\"obius strip, and hence neither is $M_i(\infty)$ for $i$ large.  
  For each such $i$, $M_i(t)$ is not a generalized M\"obius strip for all sufficiently large $t$.

Thus for all sufficiently large $i$:
\begin{align}
\label{initial-item} &\text{$M_i(0)$ is a generalized M\"obius strip, and} \\
\label{large-time-item} &\text{$M_i(t)$ is not a generalized M\"obius strip for all sufficiently large $t$.}
\end{align}

By~\eqref{initial-item}, ~\eqref{large-time-item}, and Lemma~\ref{bdry-forcing-lemma}, $M_i(\cdot)$ must have a boundary singularity.  
By~\eqref{entropy<2} and Theorem~\ref{long-theorem},
 the corresponding shrinker must be smoothly embedded and nonorientable, and must have multiplicity one.
\end{proof}

\section{The Three Pi Theorem}

\begin{theorem}\label{3-pi-theorem}
Suppose that $\Gamma'$ is a smoothly embedded, simple closed curve in $\RR^3$ 
and that $\Gamma'$ bounds a smooth minimal surface that is not a disk.  
Then $\tc(\Gamma')>3\pi$.
\end{theorem}

\begin{proof}
We may assume that $\tc(\Gamma')< 4\pi$.

By Theorem~\ref{general-redux} (applied to the family~$\Ff$ of non-disk surfaces), there is a smooth, simple closed curve $\Gamma$ 
with
\[
   \tc(\Gamma) \le \tc(\Gamma')
\]
and a
there is a standard Brakke flow $t\in\RR\mapsto M(t)$ with boundary $\Gamma$ 
for which $M(-\infty)$ is not a disk and $M(\infty)$ is a disk.

Thus the flow must have at least one singularity.  Consider the shrinker $\Sigma$ corresponding
to a tangent flow at the first singular time $T$.

If the singularity is a boundary singularity, $e(M(\cdot); \Gamma) > \frac32$ by Theorem~\ref{doubling} below.

Now suppose the singularity is an interior singularity.  Since $M(t)$ is connected and with nonempty 
boundary for $t<T$, we see that $\Sigma$ is noncompact.  
By~\cite{bernstein-wang-top}*{Corollary~1.2}, the entropy of $\Sigma$ is greater than or equal to the entropy 
\[
    \sigma_1 = (2\pi/e)^{1/2} \cong 1.52
\]
of a round cylinder.

Thus in either case (boundary singularity or interior singularity), 
\[
  e(M(T);\Gamma) \ge \frac32.
\]
On the other hand,
\[
   e(M(T);\Gamma)\le \frac{\tc(\Gamma)}{2\pi}.
\]
Thus
\[
     \frac{tc(\Gamma)}{2\pi} > \frac32,
\]
so
\[
    \tc(\Gamma')\ge   \tc(\Gamma) > 3\pi.
\]
\end{proof}

\begin{theorem}\label{doubling}
Let $\Sigma\subset \RR^{n+1}$ be a smooth, non-orientable shrinker whose boundary is an $(n-1)$-dimensional
linear subspace $L$.  Then 
\[
   e(\Sigma; L) > \frac32.
\]
\end{theorem}

\begin{proof}
Rotate $\Sigma$ by $\pi$ about $L$ to get $\Sigma'$.  Then $M:=\Sigma\cup \Sigma'$ is a smoothly immersed
shrinker without boundary.   Since it is non-orientable, it must have a point $p$ of self-intersection.  Thus  the entropy
of $M$ is $\ge 2$.  In fact, the entropy must be $>2$, since otherwise $M$ would be a cone centered at $p$.
Since $M$ is smooth, that means $M$ would be planar, which is impossible since it is nonorientable.

Thus
\[
  \int_\Sigma \frac1{(4\pi)^{1/2}} e^{-|x|^2/4}\,dx  
  = \frac12\int_M \frac1{(4\pi)^{1/2}} e^{-|x|^2/4}\,dx
  = \frac12 e(M) > 1,
\]
(where $e(M):=e(M; \emptyset)$ is the entropy of $M$),
so
\[
   e(\Sigma; L) \ge \frac12 +  \int_\Sigma\frac1{(4\pi)^{1/2}} e^{-|x|^2/4}\,dx > \frac32.
\]
\end{proof}

\section{Oriented Surfaces}

We can  improve Theorem~\ref{3-pi-theorem} slightly in the case of oriented surfaces.

\begin{theorem}\label{sharper}
Suppose that $\Gamma'$ is a smoothly embedded, simple closed curve in $\RR^3$ 
and that $\Gamma'$ bounds a smooth, oriented minimal surface that is not a disk.  
Then
\[
   \tc(\Gamma') > 2\pi \sigma_1,
\]
where $\sigma_1=(2\pi/e)^{1/2}$ is the entropy of $\Ss^1\times\RR$.
\end{theorem}

This is a slight improvement over the $3\pi$ in Theorem~\ref{3-pi-theorem} because
\[
  \frac{2\pi \sigma_1}{3\pi} =  1.01356\dots.
\]

It is conjectured that Theorem~\ref{sharper} holds with $4\pi$ in place of $2\pi\sigma_1$.
  The constant $4\pi$ would be sharp since (by work of Almgren-Thurston~\cite{AT} or 
  the simplified version by Hubbard~\cite{hubbard}), for every $\eps>0$ and $g$, there is a smooth
  simple closed curve of total curvature $<4\pi + \eps$ that bounds no embedded minimal surface of genus $\le g$.
  Such a curve bounds immersed minimal surfaces (the Douglas solutions) of each genus $\le g$, and an embedded minimal surface (the least area integral current) of genus $>g$.  See the discussion in the introduction
  of~\cite{EWW}.

\begin{proof}[Proof of Theorem~\ref{sharper}]
We may suppose that $\tc(\Gamma')< 4\pi$.

Let $\Ff$ be the family of connected, oriented surfaces of genus $\ge 1$.
By Theorem~\ref{general-redux}, there is smooth simple closed curve $\Gamma$ with
\[
    \tc(\Gamma) \le \tc(\Gamma')
\]
and a standard Brakke flow
\[
  t\in \RR \mapsto M(t)
\]
with boundary $\Gamma$ such that $M(\infty)$ is a smooth disk and such that $M(-\infty)$ is smooth, orientable
minimal surface that is not a disk.

Thus the flow must have a singularity.  
Consider the shrinker $\Sigma$ corresponding to a singularity at the first singular time $T$.
 Since $M(t)$ is orientable for $t<T$, $\Sigma$ must be orientable. 
 Thus the singularity is an interior singularity (by Theorem~\ref{long-theorem}.)

As in the proof of Theorem~\ref{3-pi-theorem}, the entropy of $\Sigma$ is greater than or equal to the
 entropy $\sigma_1$ of a cylinder,
and thus
\[
  \sigma_1 \le e(M(T); \Gamma) < \frac{tc(\Gamma)}{2\pi}.
\]
\end{proof}

\appendix

\section{Regularity and Compactness}\label{regularity-appendix}

\begin{lemma}\label{plane-lemma}
Suppose that $\Sigma$ is a $2$-dimensional integral varifold in $\RR^n$, that
\begin{equation}\label{shrinking-flow}
   t\in (-\infty,0) \mapsto |t|^{1/2} \Sigma
\end{equation}
is a standard Brakke flow without boundary, and that $\Sigma$ has entropy $<2$.
Then $\Sigma$ is a smoothly embedded surface.
\end{lemma}

\begin{proof}
First we prove:
\begin{equation}\label{cone-flat}
\text{If $\Sigma$ is a cone, then it is a multiplicity-one plane.}
\end{equation}
If $\Sigma$ is a cone, it must be a stationary cone, so its intersection with the unit sphere
is a geodesic network.  Since the entropy is $<2$, each geodesic arc occurs with multiplicity $1$.  Also, at each
vertex of the network, $3$ or fewer arcs meet.  But $3$ arcs cannot meet at a point because the flow is standard and therefore has no triple junctions.

Thus there are no vertices.  That is, $\Sigma\cap \partial \BB$ consists of disjoint, multiplicity-one geodesics.  Because the entropy is $<2$, there can only be one such geodesic.  Thus $\Sigma$ is a multiplicity-one plane. This proves~\eqref{cone-flat}.

In the general case, note that $\Sigma$ is a stationary integral varifold for the shrinker metric on $\RR^n$.  Let $C$ be a tangent 
cone to $\Sigma$ at a point $p$.  Then
\[
    t\in (-\infty,0) \mapsto C = |t|^{1/2}C
\]
is a tangent flow to the flow~\eqref{shrinking-flow} at the spacetime point $(p,-1)$.  By~\eqref{cone-flat}, $C$ is a multiplicity-one plane.  Hence (by Allard regularity) $p$ is a regular point of $\Sigma$.
\end{proof}

\begin{lemma}\label{half-plane-lemma}
Suppose that $\Sigma$ is a $2$-dimensional integral varifold in $\RR^n$, 
that
\begin{equation}\label{shrinking-bdry-flow}
  t\in (-\infty,0)\mapsto |t|^{1/2} \Sigma
\end{equation}
is a standard Brakke flow with boundary $L$, where $L$ is a straight line through the origin,
and that $e(\Sigma; L) <2$ (see Definition~\ref{entropy-def}).
Then $\Sigma$ is a smoothly embedded manifold with boundary $L$.
\end{lemma}

\begin{proof}
First we prove:
\begin{equation}\label{halfcone-flat}
\text{If $\Sigma$ is a cone, then it is a multiplicity-one half-plane.}
\end{equation}
Note that if $\Sigma$ is a cone, 
then $\Sigma\cap \partial \BB$ is a geodesic network.  Each geodesic arc occurs with multiplicity $1$.
Exactly as in the proof of Lemma~\ref{plane-lemma}, there can be no vertices of the network,
 except the two points of $L\cap \partial \BB$.

Thus $\Sigma\cap \partial \BB$ consists of
   $j$ geodesic semicircles with endpoints $L\cap\partial \BB$, 
together with
$k$ geodesic circles
(for some integers $j$ and $k$).  Consequently $\Sigma$ consists of $j$ halfplanes (each with boundary $L$) together with $k$ planes.    The extended entropy of $\Sigma$ is 
$(j/2) + k + 1/2$,
so 
\begin{equation}\label{fractions}
  \frac{j+1}2 + k < 2
\end{equation}
 and therefore $j<3$.  By standardness, the mod $2$ boundary of $\Sigma$ is $L$, so $j$ is odd.
Thus $j=1$.  By~\eqref{fractions}, $k=0$.  This proves~\eqref{halfcone-flat}.

Now we consider the general case.   Exactly as in Lemma~\ref{plane-lemma}, $\Sigma$ is smooth and embedded except perhaps along $L$.   Let $C$ be a tangent cone to $\Sigma$ at a point $p\in L$.   Then
\[
   t\in(-\infty,0) \mapsto C = |t|^{1/2}C
\]
is a tangent flow to the flow~\eqref{shrinking-bdry-flow}
 at the spacetime point $(p,-1)$, so $C$ is a multiplicity-one half-plane by~\eqref{halfcone-flat}. Thus $p$
is a regular point of $\Sigma$ by Allard regularity.
\end{proof}

\begin{corollary}\label{half-plane-corollary}
Let $\Sigma$ be as in Lemma~\ref{half-plane-lemma}.  If $\Sigma$ is invariant under translations in some direction,
 then $\Sigma$ is a half-plane.
\end{corollary}

\begin{proof}
The direction of translational invariance would have to be $L$.  Let $P$ be the tangent half-plane to $\Sigma$ at $0$.
Then $P$ and $\Sigma$ are $g$-minimal surfaces that are tangent along $L$ (where $g$ is the shrinker metric.)
 Thus $P$ and $\Sigma$ coincide.
\end{proof}

\begin{proposition}\label{most-times-proposition}
Suppose that $\Gamma$ is a smoothly embedded curve in $\RR^n$ and suppose that
\[
 t\in I \mapsto M(t)
\]
is a standard Brakke flow with boundary $\Gamma$ satisfying the entropy bound
\[
   e(M(t);\Gamma) < 2 \quad\text{for all $t$}.
\]
Then the set of singular times has measure $0$.
\end{proposition}

\begin{proof}
Let $\Sigma$ be the shrinker for the tangent flow at a singularity.  Then $\Sigma$ is smoothly embedded,
so if it had a direction of translational invariance, then by Corollary~\ref{half-plane-corollary}, the singularity
would not be at a boundary point.  Thus $\Sigma$ would be a cylinder.
Regularity at almost all times follows from the standard stratification theory~\cite{White-Stratification}.
\end{proof}

\begin{remark}
From the proof, we see that we do not really need to assume that the surfaces have entropy $<2$; it is enough
to assume that all the tangent flows have entropy $<2$.
\end{remark}

\begin{theorem}\label{minimal-compact}
Suppose that $\Gamma$ is a smooth, simple closed curve in $\RR^n$ of total curvature $\le 4\pi$ and that $\Gamma_i$
converges smoothly to $\Gamma$.  Suppose that $M_i$ is a smooth embedded minimal surface bounded by $\Gamma_i$.
Then, after passing to a subsequence, $M_i$ converges smoothly to a minimal surface bounded by $\Gamma$.
\end{theorem}

\begin{proof}
This is not hard to prove by minimal surface techiques (using extended monotonicity~\cite{EWW}*{Theorem~9.1} and
arguments analogous to the proofs of Lemmas~\ref{plane-lemma} and~\ref{half-plane-lemma}).   Here we prove it 
using mean curvature flow since we have already established all the necessary ingredients.

By the convex hull property of minimal surfaces and the isoperimetric inequality, 
the $M_i$ lie in a compact subset of $\RR^n$ and their areas areas are uniformly bounded.
Thus, after passing to a subsequence, the $M_i$ converge weakly to a compactly supported integral varifold $M$.
Note that the standard Brakke flows
\[
   t\in\RR \mapsto M_i
\]
with boundary $\Gamma_i$
converge to the Brakke flow
\begin{equation}\label{M-flow}
  t\in \RR \mapsto M
\end{equation}
with boundary $\Gamma$.
Thus $t\in \RR\mapsto M$ is also standard~\cite{white-mcf-boundary}*{Theorem~13.1}.
By Proposition~\ref{most-times-proposition}, the flow~\eqref{M-flow} is smooth at almost all times. 
Thus $M$ is smooth.  By Allard Regularity, the convergence is smooth.
\end{proof}

\section{Reducing Total Curvature}\label{deformation-appendix}

\begin{theorem}\label{first-deformation}
Let $\alpha< 4\pi$.  Let $C(\alpha)$ be the collection of simple closed, polygonal curves in $\RR^n$ with total curvature at most $\alpha$.  Then $C(\alpha)$ is connected.
\end{theorem}

\begin{proof}
It suffices to show that any curve in $C(\alpha)$ can be deformed through curves in $C(\alpha)$
to a triangle.

Let $\Gamma$ be a curve in $C(\alpha)$.
According to Milnor's Theorem~\cite{milnor},  we can assume (by rotating and scaling) that the image of height function 
\[
  h: x\in \Gamma \mapsto x\cdot \ee_n
\]
is $[0,1]$, and that for each $y\in (0,1)$, there are exactly two points $\gamma_1(y)$ and $\gamma_2(y)$ of $\Gamma$ at which
$h=y$.

We may also assume that $h=1$ at exactly $1$ point.  

Let $\Gamma(0)=\Gamma$, and for $t\in (0,1)$, let $\Gamma(t)$ be
$\Gamma\cap \{h\ge t\}$ together with the segment joining $\gamma_1(t)$ and $\gamma_2(t)$.

The total curvature of $\Gamma(t)$ is a decreasing function of $t$, so $\Gamma(t)\in C(\alpha)$
for all $t\in [0,1)$.  Note that for $t$ close to $1$, $\Gamma(t)$ is a triangle.
\end{proof}

\begin{lemma}\label{second-deformation}
Let $\gamma:[0,1]\to \RR^n$ be a smooth simple closed curve.
Then there exists a one-parameter family 
\[
  t\mapsto \gamma_t
\]
such that 
\begin{enumerate}
\item $\gamma_0=\gamma$.
\item $\gamma_1$ is polygonal.
\item Each $\gamma_t$ is a piecewise-smooth simple closed curve.
\item The total curvature of $\gamma_t$ is a decreasing function of $t$.
\item For each $t\in [0,1]$ and each $s\in [0,1]$, 
\[
   \gamma'_t(s_-)\cdot \gamma'_t(s_+) > 0.
\]
\end{enumerate}
\end{lemma}

\begin{proof}
Let $\gamma:[0,1]\to \Gamma$ be a smooth parametrization.

Fix a large integer $N$, and for each $t\in [0,1]$, let $\gamma_t:[0,1]\to \RR^n$ be
 the closed curve formed from $\gamma$ by replacing (for each $k=0,\dots, N-1$)
\[
    \gamma | [k/N, (k+t)/N]
\]
by the line segment with the same endpoints.

Then $\gamma_0=\gamma$, the total curvature of $\gamma_t$ is a decreasing function of $t$
(by~\cite{milnor}),
and $\gamma_1$ is polygonal.

Note also that if $N$ is sufficiently large, then each $\gamma_t$ will be an embedding.
\end{proof}

\begin{theorem}\label{deformation-theorem}
Let $\Gamma$ be a smooth, simple closed curve of total curvature $\le \alpha<4\pi$.
Then $\Gamma$ can be deformed among such curves to a planar convex curve.
\end{theorem}

\begin{proof}
If $C$ is a simple closed curve in $\RR^n$, let $\phi_t(C)$ be the result of flowing $C$ for time $t$ by curve-shortening flow.  (We assume $t>0$ is small enough that the flow is smooth on 
the time interval $(0,t]$.)

By Theorem~\ref{first-deformation} and Lemma~\ref{second-deformation}, there is a one-parameter family
\[
    s\in [0,2] \mapsto \Gamma_s
\]
of simple closed curves
such that:
\begin{enumerate}
\item $\Gamma_0=\Gamma$.
\item $s\in [0,1]\mapsto \Gamma_s$ are piecewise smooth curves as in Lemma~\ref{second-deformation}.
\item $\Gamma_s$ is polygonal for $s\in [1,2]$.
\item $\Gamma_2$ is a triangle.
\item The total curvature of $\Gamma_s$ is a decreasing function of $s$.
\end{enumerate}

Choose $\eps>0$ small enough so that for each $s\in [0,2]$, then curve-shortening flow
starting with $\Gamma_s$ remains smooth and embedded on the time interval $(0,\eps]$.

Now deform $\Gamma=\Gamma_0$ to $\phi_\eps(\Gamma_0)$ by
\[
   t\in [0,\eps] \mapsto \phi_t(\Gamma_0).
\]
Then deform $\phi_\eps(\Gamma_0)$ to the plane convex curve $\phi_\eps(\Gamma_2)$ by
\[
    s\in [0,2] \mapsto \phi_\eps(\Gamma_s).
\]
Note that we have deformed $\Gamma$ to the plane convex curve $\phi_\eps(\Gamma)$
through smooth, simple closed curves, each of total curvature $\le \alpha$.  
(By~\cite{hatt}*{Lemma~3.4}, curve shortening in $\RR^n$ reduces total curvature.)
\end{proof}

\section{Unique Limits}

\begin{theorem}\label{unique}
Suppose $\Sigma$ is an $m$-dimensional compact, smoothly embedded minimal surface in $\RR^n$
with smooth boundary $\Gamma$.
Suppose
\[
   t\in I \mapsto M(t)
\]
is a standard mean curvature flow of $m$-dimensional surfaces in $\RR^n$ with fixed boundary $\Gamma$.
\begin{enumerate}[\upshape (1)]
\item\label{forward} If $I=[0,\infty)$ and if there is a sequence $t_i\to \infty$ such that $M(t_i)$ converges smoothly (with multiplicity $1$)
to $\Sigma$, 
then $M(t)$ converges smoothly as $t\to\infty$ to $\Sigma$.
\item\label{backward} If $I=(-\infty, 0]$ and  if there is a sequence $t_i\to -\infty$ such that $M(t_i)$ converges smoothly (with multiplicity $1$)
to $\Sigma$, then 
Then $M(t)$ converges smoothly as $t\to-\infty$ to $\Sigma$.
\end{enumerate}
\end{theorem}

\begin{proof}
If $t\in [0,\infty)\mapsto M(t)$ is a solution of the renormalized mean curvature flow
\[
   (\textnormal{velocity at $x$})^\perp = H + \frac{x^\perp}2,
\]
and if there is no boundary, then Assertion~\ref{forward} is proved in~\cite{schulze-unique}*{Corollary~1.2}.
In fact, the same proof works when $M$ has boundary, when $t\to -\infty$, and 
when the flow is mean curvature flow rather than renormalized mean curvature
flow.  (For mean curvature flow, in the proofs in~\cite{schulze-unique}, one should let $\Ee$ be the ordinary area
of the surface and $\rho$ be the constant function $1$.)
\end{proof}

\end{document}